\documentclass[a4paper,11pt,reqno]{amsart}
\usepackage{amssymb,amsmath,mathrsfs,graphics,graphicx,mathptm,float}
\usepackage[T1]{fontenc}
\usepackage[english, francais]{babel}
\usepackage[all]{xy}
\theoremstyle{plain}
\newtheorem{thm}{Th\'eor\`eme}[section]
\newtheorem{pro}[thm]{Proposition}

\newtheorem{cor}[thm]{Corollaire}

\theoremstyle{definition}

\newtheorem{rem}[thm]{Remarque}

\def\og{\leavevmode\raise.3ex\hbox{$\scriptscriptstyle\langle\!\langle$~}}
\def\fg{\leavevmode\raise.3ex\hbox{~$\!\scriptscriptstyle\,\rangle\!\rangle$}}

\usepackage[colorlinks, linktocpage, 
citecolor = blue, linkcolor = blue]{hyperref}

\addtocounter{section}{0}             
\numberwithin{equation}{section}       

\title{Formes logarithmiques et feuilletages non dicritiques}

\author{Dominique \textsc{Cerveau}}

\address{Membre de l'Institut Universitaire de France.
IRMAR, UMR 6625 du CNRS, Universit\'e de Rennes $1$, $35042$ Rennes, France.}
\email{dominique.cerveau@univ-rennes1.fr}

\begin{document}
\selectlanguage{french}

\maketitle\begin{center}{\today}\end{center}

\begin{flushright}
{\sl \`A Xavier Gomez Mont}\\

{\sl jeune math\'ematicien \`a}\\

{\sl l'humour inclassable}
\end{flushright}

\begin{abstract}
Pour un feuilletage alg\'ebrique $\mathcal F$ de codimension 1 sur l'espace projectif $\mathbb P_{\mathbb C}^{n}$, sous des conditions raisonnables portant sur la nature des singularit\'es, le degr\'e des hypersurfaces alg\'ebriques invariantes est major\'e par $d+2$ o\`u $d$ est le degr\'e de $\mathcal F$ (Carnicer \cite{5}, Cerveau-Lins-Neto \cite{6}). On s'int\'eresse ici au cas extr\'emal o\`u le degr\'e d'une telle hypersurface est pr\'ecis\'ement $d+2$ compl\'etant en cel\`a des r\'esultats de Brunella \cite{1, 6}.

\end{abstract}

\tableofcontents

 \section{Introduction} 
 
 Soit $X$une vari\'et\'e complexe ; si $\omega$ est une $1$-forme diff\'erentielle m\'eromorphe sur $X$ on note $D=\mathrm{Pol}\,\omega$ son diviseur de p\^oles. On dit que $\omega$ est une forme logarithmique si $\omega$ et $d\omega$ sont \`a p\^oles simples le long de $D$. On sait qu'une $1$-forme holomorphe sur une vari\'et\'e projective, et plus g\'en\'eralement sur une vari\'et\'e k\"ahl\'erienne, est ferm\'ee. C'est une cons\'equence de la formule de Stockes. Le r\'esultat qui suit g\'en\'eralise ce fait ; il est du \`a P. Deligne :
   
\begin{thm}[\cite{7}] \label{thm:deligne}
Soient $X$ une vari\'et\'e projective complexe et $\omega$ une $1$-forme logarithmique \`a p\^oles le long du diviseur $D$. Si les singularit\'es de $D$ sont des croisements ordinaires, la forme $\omega$ est ferm\'ee.
\end{thm}

On peut en fait all\'eger les hypoth\`eses puisqu'une $1$-forme m\'eromorphe sur $X$ est ferm\'ee d\`es qu'elle l'est en restriction \`a toute section hyperplane g\'en\'erale de dimension au moins 2. De sorte qu'il suffit de supposer que les singularit\'es de $D$ sont des croisements ordinaires en dehors d'un ensemble de codimension 3 de $X$. En un certain sens le Th\'eor\`eme \ref{thm:deligne} est de nature 2-dimensionnelle.

Comme l'ont remarqu\'e plusieurs auteurs, ce r\'esultat est directement li\'e au probl\`eme de l'estimation du degr\'e des hypersurfaces invariantes des feuilletages de codimension $1$ sur les vari\'et\'es projectives. Les premiers r\'esultats en ce sens sont dus \`a Carnicer \cite{5} et Cerveau-Lins Neto \cite{6} dans le cadre des feuilletages alg\'ebriques du plan. Le th\'eor\`eme de Carnicer n\'ecessite des hypoth\`eses sur la nature des points singuliers du feuilletage: 

\begin{thm}[\cite{5}] \label{thm:inegalite}
Soit $\mathcal F$ un feuilletage de degr\'e $d$, de codimension 1 sur l'espace projectif $\mathbb P_{\mathbb C}^{2}$. On suppose que $\mathcal F$ poss\`ede une courbe alg\'ebrique invariante $S$ de degr\'e $m$. Si les points singuliers de $\mathcal F$ situ\'es sur $S$ sont non dicritiques alors $m\leq d+2$.
\end{thm}

\`A l'inverse dans ce qui suit les hypoth\`eses portent sur les points singuliers de $S$ et non sur ceux de $\mathcal F$: 

\begin{thm}[\cite{6}]\label{thm:egalite}
Soit $\mathcal F$ un feuilletage de degr\'e $d$ du plan $\mathbb P_{\mathbb C}^{2}$ poss\'edant une courbe alg\'ebrique invariante~$S$ de degr\'e $m$. Si les singularit\'es de $S$ sont des croisements ordinaires alors $m\leq d+2$. Lorsque l'\'egalit\'e est r\'ealis\'ee, $m=d+2$, le feuilletage $\mathcal F$ est d\'efini par une forme ferm\'ee logarithmique.
\end{thm}

C'est la derni\`ere partie de l'\'enonc\'e qui est directement reli\'ee \`a celui de Deligne. Le Th\'eor\`eme \ref{thm:egalite} se g\'en\'eralise stricto sensu aux feuilletages de codimension $1$ sur $\mathbb P_{\mathbb C}^{n}$ ayant une hypersurface invariante \`a croisements ordinaires. Mieux Brunella et Mend\`es \'etablissent dans \cite{2} un r\'esultat plus g\'en\'eral concernant les champs d'hyperplans (\`a priori non n\'ecessairement int\'egrables) ayant encore une hypersurface \`a croisements normaux et ce sur les vari\'et\'es projectives ayant $\mathbb Z$ comme groupe de Picard.

Dans cet article on pr\'ecise le Th\'eor\`eme \ref{thm:inegalite} dans le cas extr\'emal o\`u l'in\'egalit\'e est une \'egalit\'e : $m=d+2$ ;  pour cel\`a on relie les concepts de non dicriticit\'e et de formes logarithmiques (Propositions \ref{pro:pro1}, \ref{pro:pro2}, \ref{pro:pro3}). Dans \cite[Proposition 10]{1} Brunella pr\'esente un r\'esultat similaire en utilisant des arguments d'indice de champs de vecteurs.

\begin{thm} 
Soit $\mathcal F$ un feuilletage de degr\'e $d$ sur l'espace projectif $\mathbb P_{\mathbb C}^{2}$ poss\'edant une courbe alg\'ebrique invariante $S$ de degr\'e pr\'ecis\'ement $d+2$. Si les points singuliers de $\mathcal F$ sur $S$ sont non dicritiques, alors $\mathcal F$ est donn\'e par une forme ferm\'ee logarithmique \`a p\^oles le long de $S$.
\end{thm}

On adapte ensuite cet \'enonc\'e aux dimensions sup\'erieures et on donne quelques applications.

\section{Formes logarithmiques et r\'esolution des singularit\'es}

Soit $\mathcal F$ un germe de feuilletage singulier \`a l'origine de $\mathbb C^{2}$ ; on note $\omega=A\,dx+ B\, dy$ un germe de $1$-forme \`a singularit\'e isol\'ee en 0 d\'efinissant $\mathcal F$. Un tel $\omega$ est d\'efini \`a unit\'e multiplicative pr\`es. Par d\'efinition la \textbf{multiplicit\'e alg\'ebrique} ou l'\textbf{ordre de $\mathcal F$ en $0$} est l'entier
$$\nu(\mathcal F)=\nu(\omega)=\inf (\nu(A), \nu(B))$$
o\`u $\nu(A)$ et $\nu(B)$ d\'esignent les ordres des fonctions holomorphes $A$ et $B$ en $0$. Soit $S$ un germe de courbe d'\'equation r\'eduite $f=0$ \`a l'origine de $\mathbb C^{2}$. On dit que $S$ est une \textbf{s\'eparatrice} ou une \textbf{courbe invariante} de $\mathcal F$ si $S\smallsetminus\{0\}$ est une feuille (au sens ordinaire) du feuilletage r\'egulier $\mathcal F|_{\mathbb C^2\smallsetminus\{0\},0}$. Ceci se traduit en termes alg\'ebriques par : la $2$-forme $\omega\wedge df$ est divisible par $f$, \emph{i.e.} s'annule identiquement sur $S$.

Le germe de feuilletage $\mathcal F$ est dit \textbf{non dicritique} s'il ne poss\`ede qu'un nombre fini de s\'eparatrices. Il en poss\`ede d'ailleurs au moins une d'apr\`es un \'enonc\'e c\'el\`ebre de Camacho et Sad \cite{4}. Lorsque $\mathcal F$ poss\`ede une infinit\'e de s\'eparatrices il est donc dit \textbf{dicritique}. Cette notion de dicriticit\'e s'interpr\`ete en termes de r\'eduction des singularit\'es. Soit $\pi \colon\widetilde{\mathbb C^2}\rightarrow \mathbb C^2,0$ une r\'eduction des singularit\'es de $\mathcal F$ ; alors $\mathcal F$ est non dicritique si et seulement si chaque composante du diviseur exceptionnel $\pi^{-1}(0)$ est invariante par le  feuilletage transform\'e strict $\pi^{-1}(\mathcal F)$ de $\mathcal F$ par $\pi$. On rappelle qu'en dimension $2$ toutes les r\'eductions (minimales) sont isomorphes de sorte que le discours qui pr\'ec\`ede est ind\'ependant du choix de la r\'eduction. La notion de forme logarithmique se localise sans probl\`eme : le germe de $1$-forme m\'eromorphe $\Omega$ \`a l'origine de $\mathbb C^2$ est \textbf{logarithmique} si $\Omega$ et $d\Omega$ sont \`a p\^oles simples. L'\'enonc\'e qui suit est \'el\'ementaire:

\begin{pro} \label{pro:pro1}
Soient $\mathcal F$ un germe de feuilletage \`a l'origine de $\mathbb C^2$ d\'efini par la $1$-forme holomorphe $\omega$ et~$S$ une courbe invariante $($pas n\'ecessairement irr\'eductible$)$ de $\mathcal F$, d'\'equation r\'eduite $f=0$. Alors la forme m\'eromorphe $\Omega=\omega/f$ est logarithmique.
\end{pro}

\begin{proof}
Puisque $f$ est r\'eduite $\Omega$ est \`a p\^oles simples.
Maintenant $d\Big(\frac{\Omega}{f}\Big)=\frac{d\omega}{f}+\frac{\omega\wedge df}{f^2}$ est aussi \`a p\^oles simples puisque $\omega \wedge df$ est divisible par $f$.  
\end{proof}

Le fait pour une $1$-forme d'\^etre logarithmique n'est pas "invariant" par \'eclatement. Par exemple la $1$-forme $\Omega=\frac{x\,dy-y\,dx}{x^4+y^4}$ est logarithmique, mais si on \'eclate l'origine par la transformation quadratique $$\sigma\colon (x,t)\rightarrow (x,tx)$$
alors $\sigma^* \Omega=\frac{dt}{x^2 (1+t^4)}$ est \`a p\^ole double le long du diviseur exceptionnel $x=0$. C'est en fait un avatar de la dicriticit\'e du feuilletage radial associ\'e \`a $x\,dy-y\,dx$.

Par contre dans le cas non dicritique on a la: 

\begin{pro} \label{pro:pro2}
Soient $\mathcal F$ un germe de feuilletage non dicritique et $S$ une courbe invariante par $\mathcal F$. Soient $f=0$ une \'equation r\'eduite de $S$ et $\omega$  une $1$-forme holomorphe d\'efinissant $\mathcal F$. Si $\sigma$ est l'application d'\'eclatement de l'origine, alors la $1$-forme m\'eromorphe $\sigma^{*} \Big(\frac{\omega}{f}\Big)$ est logarithmique.
\end{pro}

\begin{proof}
Elle repose sur l'in\'egalit\'e suivante \cite{3} \'etablie dans le cas non dicritique pr\'ecis\'ement : $\nu(f)\leq \nu(\omega)+1$. La proposition est alors une simple v\'erification que l'on effectue par exemple dans la carte $(x,t)$ o\`u $\sigma(x,t)=(x,tx)$. On a :
$$\sigma^* \Big(\frac{\omega}{f}\Big)=\frac{x^{\nu(\omega)} \widetilde{\omega}}{x^{\nu(f)} \widetilde{f}}$$
avec $\widetilde{\omega}$ et $\widetilde{f}$ holomorphes. L'in\'egalit\'e $\nu(f)\leq \nu(\omega)+1$ implique que $\sigma^{*}(\frac{\omega}{f})$est au pire \`a p\^ole simple le long de $x=0$ ; le comportement de $\sigma^*\Big(\frac{\omega}{f}\Big)$ le long de $\widetilde{f}=0$ est bien entendu le m\^eme que celui de $\omega$ le long de $f=0$. Comme dans le cas non dicritique le diviseur exceptionnel $x=0$ est invariant par le feuilletage $\sigma^{*} \mathcal F$ d\'efini par $\widetilde{\omega}$, la $2$-forme $d\Big(\sigma^*\big(\frac{\omega}{f}\big)\Big)$ est aussi \`a p\^oles au pire simples. 
\end{proof}

\begin{rem} 
Il se peut, et c'est le cas si $\nu(f)\leq \nu(\omega)$, que $\sigma^{*} (\frac{\omega}{f})$ n'ait pas de p\^ole le long du diviseur $x=0$.
\end{rem}

\begin{pro} \label{pro:pro3}
Soit $\mathcal F$un germe de feuilletage non dicritique \`a l'origine de $\mathbb C^2$ donn\'e par la $1$-forme $\omega$. Soit $S=(f=0)$ une courbe invariante par $\mathcal F$, avec $f$ r\'eduite. Soit $\pi : \widetilde{\mathbb C}^2\rightarrow \mathbb C^2,0$ la r\'esolution des singularit\'es de $\mathcal F$. Alors la $1$-forme $\pi^{*}\Big(\frac{\omega}{f}\Big)$ m\'eromorphe sur $\widetilde{\mathbb C}^2$ est logarithmique.
\end{pro}

\begin{proof}
Comme l'application $\pi$ est une composition finie d'\'eclatements et que la notion de non dicriticit\'e compatible aux \'eclatements, c'est une application directe de la Proposition \ref{pro:pro2}.
\end{proof}

\section{D\'emonstration du th\'eor\`eme 4.}

Le feuilletage $\mathcal F$ est donn\'e en coordonn\'ees homog\`enes $(z_{0} : z_{1} : z_{2})$ par une $1$-forme
$$\omega=A_{0} dz_{0}+ A_{1} dz_{1}+ A_{2}dz_{2}$$
o\`u les $A_{i}$ sont des polyn\^omes homog\`enes de degr\'e $d+1$, $\mathrm{pgcd}\,(A_{0}, A_{1}, A_{2})=1$ satisfaisant l'identit\'e d'Euler : 
$$\sum z_{i} A_{i}=0.$$
La courbe invariante $S$ est elle donn\'ee par un polyn\^ome homog\`ene r\'eduit $f$ de degr\'e $d+2$. La $1$-forme m\'eromorphe $\frac{\omega}{f}$ est donc invariante par les homoth\'eties $z\mapsto t.z$. En utilisant l'identit\'e d'Euler on constate qu'elle d\'efinit une $1$-forme m\'eromorphe $\Omega$ sur $\mathbb P_{\mathbb C}^2$ \`a p\^oles simples le long de $S$. Comme $S$ est invariante par~$\mathcal F$ la forme $\Omega$ est donc logarithmique \`a p\^oles le long de $S$. On consid\`ere la r\'eduction des singularit\'es
$$\pi\colon  \widetilde{\mathbb P^2_{\mathbb C}}\rightarrow \mathbb P^2_{\mathbb C}$$
du feuilletage $\mathcal F$. Comme en chaque point singulier $p\in S$ le feuilletage $\mathcal F_{,p}$ est non dicritique on peut appliquer la Proposition \ref{pro:pro3}; ainsi $\pi^{*} \Omega$ est logarithmique sur $\widetilde{\mathbb P^2}_{\mathbb C}$. Son diviseur de p\^oles est contenu dans le transform\'e total $\pi^{-1} (S)$ de $S$ par $\pi$ (c'est l'union des diviseurs exceptionnels et de la transform\'ee stricte de~$S$). Comme $\pi^{-1} (S)$ est \`a croisements ordinaires, les p\^oles de $\pi^{*} (\Omega)$ le sont aussi et le th\'eor\`eme de Deligne affirme que $\pi^{*} \Omega$ est ferm\'ee; par suite $\Omega$ aussi.

\section{Applications et g\'en\'eralisation}

Comme l'aura not\'e le lecteur on obtient de mani\`ere analogue et directe le : 

\begin{thm} \label{thm:thm1}
Soient $X$ une surface projective et $\omega$ une $1$-forme logarithmique sur $X$. Si les singularit\'es du feuilletage associ\'e \`a $\omega$ situ\'ees sur le diviseur des p\^oles de $\omega$ sont non dicritiques, alors la forme $\omega$ est ferm\'ee.
\end{thm}

En fait le r\'esultat pr\'ec\'edent se g\'en\'eralise en toute dimension :

\begin{thm}  \label{thm:thm2}
Soient $X \subset \mathbb P^n_{\mathbb C}$ une vari\'et\'e projective et $\omega$ une $1$-forme logarithmique sur $X$. On suppose que dans une famille g\'en\'erique de sections lin\'eaires de dimension $n-\dim X+2$ les hypoth\`eses du Th\'eor\`eme \ref{thm:thm1} sont r\'ealis\'ees. Alors la forme $\omega$ est ferm\'ee.
\end{thm}

Dans l'esprit du Th\'eor\`eme  \ref{thm:thm2} nous avons le :

\begin{cor} 
Soit $\mathcal F$ un feuilletage de codimension 1 sur l'espace $\mathbb P^{n}_{\mathbb C}, n\geq 2$. Si dans une section 2-plane g\'en\'erale $i : \mathbb P^2_{\mathbb C} \rightarrow \mathbb P^{n}_{\mathbb C}$ le feuilletage restreint $i^{*} \mathcal F$ satisfait les hypoth\`eses du Th\'eor\`eme \ref{thm:thm1} alors $\mathcal F$ est d\'efini par une forme ferm\'ee logarithmique. En particulier $\mathcal F$ poss\`ede une hypersurface invariante de degr\'e $\deg \mathcal F+~2$.
\end{cor}

\section{Compl\'ements}

Consid\'erons sur $\mathbb P^2_{\mathbb C}$ le feuilletage $\mathcal F$ donn\'e en carte affine $(z_1, z_2)$ par la $1$-forme :
$$\omega=z_1 dz_2-z_2dz_1+z_1 z_2(z_2-z_1)\left(\alpha \frac{dz_1}{z_1}+ \beta \frac{dz_2}{z_2}+ \gamma \frac{d(z_2-z_1)}{z_2-z_1}\right)$$
 o\`u les $\alpha$, $\beta$, $\gamma$ sont des constantes complexes.
 
 C'est un feuilletage de degr\'e $2$, ayant une singularit\'e dicritique en l'origine. Il poss\`ede les droites invariantes $z_1=0$, $z_2=0$, $z_1=z_2$ et la droite \`a l'infini (tout du moins lorsque $\alpha+\beta+\gamma \neq 0$) et par cons\'equent une s\'eparatrice r\'eduite de degr\'e $4$. On d\'emontre facilement par des arguments holonomiques que pour $\alpha$, $\beta$, $\gamma$ g\'en\'eriques $\mathcal F$ n'est pas donn\'e par une $1$-forme ferm\'ee. On note aussi que $\frac{\omega}{z_1 z_2(z_2-z_1)}$ d\'efinit sur $\mathbb P^2_{\mathbb C}$ une $1$-forme logarithmique \`a p\^oles le long des $4$ droites ci-dessus. Par contre si l'on \'eclate l'origine, la forme \'eclat\'ee n'est pas logarithmique le long du diviseur exceptionnel.

Dans leur \'etude des feuilletages modulaires de Hilbert \cite{9}, Mendes et Pereira donnent l'exemple d'un feuilletage quadratique de $\mathbb P^2_{\mathbb C}$, non d\'efini par une forme ferm\'ee, et poss\'edant une courbe invariante irr\'eductible de degr\'e $S=\deg \mathcal F+3$. Ce feuilletage est "transversalement projectif" et  "non transversalement affine". Le lecteur int\'eress\'e pourra consulter l'article \cite{8} de Lins Neto o\`u l'auteur examine des familles de feuilletages de petit degr\'e sur $\mathbb P^2_{\mathbb C}$ ayant  des courbes invariantes de degr\'e grand. Dans l'esprit de cet article on peut se demander si un feuilletage $\mathcal F$ de $\mathbb P^2_{\mathbb C}$ poss\`edant une courbe invariante de degr\'e "tr\`es grand" relativement \`a celui de $\mathcal F$ est transversalement projectif.

Terminons par les remarques suivantes ; si un feuilletage $\mathcal F$ de $\mathbb P^2_{\mathbb C}$ a une courbe invariante de $S$ de degr\'e pr\'ecis\'ement $\deg \mathcal F+2$ les singularit\'es de $\mathcal F$ sur $S$ \'etant non dicritiques alors $S$ a au moins 3 composantes irr\'eductibles.

Dans le m\^eme ordre d'id\'ee soient $\mathcal F$ un feuilletage de degr\'e $d$ sur $\mathbb P^{n}_{\mathbb C}$ et $H$ une hypersurface invariante de $\mathcal F$. Si les singularit\'es de $H$ sont de codimension sup\'erieure o\`u \'egale \`a $3$ alors degr\'e $H\leq d+1$. En effet si $\omega$ est une $1$-forme homog\`ene d\'efinissant $\mathcal F$ et $h$ un  polyn\^ome homog\`ene irr\'eductible tel que $H=(h=0)$ alors $\omega=a\,dh+h \eta$ avec $a\in \mathcal O(\mathbb C^{n+1})$ et $\eta \in \Omega^1 (\mathbb C^{n+1})$ homog\`enes ; c'est une cons\'equence du lemme de division de Rham-Saito. L'identit\'e d'Euler $i_{R}\omega=0$, o\`u $R$ d\'esigne le champ radial $R=\displaystyle\mathop{\Sigma}^{n}_{i=0} z_{i} \frac{\partial}{\partial z_{i}}$, implique alors que :
$$h \Big((\deg h) a+i_{R} \eta\Big)=0$$
en particulier $\deg a\geq 1$ et $\deg H\leq d+1$. Dans le cas extr\'emal o\`u l'in\'egalit\'e est une \'egalit\'e, $\deg H=d+1$, on constate que $\eta=-\delta\, da$ o\`u $\delta=\deg h$, de sorte que $h/a^\delta$ est une int\'egrale premi\`ere rationnelle de $\mathcal F$. Notons que dans la carte affine $a=1$ le feuilletage $\mathcal F$ a une int\'egrale premi\`ere polynomiale. Remarquons que dans ce cas le feuilletage $\mathcal F$ a des singularit\'es dicritiques le long de $H$.

\bigskip

Une pens\'ee pour Marco Brunella qui s'est beaucoup int\'eress\'e \`a ce type de probl\`emes. Je tiens \`a  remercier Julie D\'eserti pour son aide constante et d\'esint\'eress\'ee.


\begin{thebibliography}{CLNS84}

\bibitem[BM00]{2}
M.~Brunella and L.~G. Mendes.
\newblock Bounding the degree of solutions to {P}faff equations.
\newblock {\em Publ. Mat.}, 44(2):593--604, 2000.

\bibitem[Bru97]{1}
M.~Brunella.
\newblock Some remarks on indices of holomorphic vector fields.
\newblock {\em Publ. Mat.}, 41(2):527--544, 1997.

\bibitem[Car94]{5}
M.~Carnicer.
\newblock The {P}oincar\'e problem in the nondicritical case.
\newblock {\em Ann. of Math. (2)}, 140(2):289--294, 1994.

\bibitem[CLN91]{6}
D.~Cerveau and A.~Lins~Neto.
\newblock Holomorphic foliations in {${\bf C}{\rm P}(2)$} having an invariant
  algebraic curve.
\newblock {\em Ann. Inst. Fourier (Grenoble)}, 41(4):883--903, 1991.

\bibitem[CLNS84]{3}
C.~Camacho, A.~Lins~Neto, and P.~Sad.
\newblock Topological invariants and equidesingularization for holomorphic
  vector fields.
\newblock {\em J. Differential Geom.}, 20(1):143--174, 1984.

\bibitem[CS82]{4}
C.~Camacho and P.~Sad.
\newblock Invariant varieties through singularities of holomorphic vector
  fields.
\newblock {\em Ann. of Math. (2)}, 115(3):579--595, 1982.

\bibitem[Del71]{7}
P.~Deligne.
\newblock Th\'eorie de {H}odge. {II}.
\newblock {\em Inst. Hautes \'Etudes Sci. Publ. Math.}, (40):5--57, 1971.

\bibitem[LN02]{8}
A.~Lins~Neto.
\newblock Some examples for the {P}oincar\'e and {P}ainlev\'e problems.
\newblock {\em Ann. Sci. \'Ecole Norm. Sup. (4)}, 35(2):231--266, 2002.

\bibitem[MP05]{9}
L.~G. Mendes and J.~V. Pereira.
\newblock Hilbert modular foliations on the projective plane.
\newblock {\em Comment. Math. Helv.}, 80(2):243--291, 2005.

\end{thebibliography}
\end{document}